\documentclass{amsart}
\usepackage{amsmath, amssymb, amsthm, overpic}

\usepackage[margin=1.25in]{geometry}
\newcommand{\Ls}{\mathcal{L}}
\newcommand{\I}{\mathcal{I}}
\DeclareMathOperator{\Teich}{Teich}
\newcommand{\ms}{\mathcal{M}}
\newcommand{\st}{\;|\;}
\newcommand{\N}{\mathbb{N}}

\newcommand{\Z}{\mathbb{Z}}
\DeclareMathOperator{\Isom}{Isom}

\DeclareMathOperator{\arcsinh}{arcsinh}
\DeclareMathOperator{\Ends}{Ends}

\newtheorem{thm}{Theorem}[section]
\newtheorem{lemma}[thm]{Lemma}

\newtheorem{prop}[thm]{Proposition}

\newtheorem{rmk}[thm]{Remark}

\newtheorem{theorintro}{Theorem}

\newtheorem*{thm:finitediscretesunada}{Theorem \ref{thm:finitediscretesunada}}
\newtheorem*{thm:finitegroup}{Theorem \ref{thm:finitegroup}}

\usepackage[most]{tcolorbox}
\definecolor{mygreen}{RGB}{0, 160, 0}
\usepackage{hyperref} 
\hypersetup{colorlinks=true,
allcolors=mygreen}

\title[Isospectrality and isometry groups]{Isospectrality and isometry groups  for infinite-type hyperbolic surfaces with discrete length spectrum}
\date{\today}


\author{Federica Fanoni}
\email{federica.fanoni@u-pec.fr}
\address{CNRS, Univ Paris Est Creteil, Univ Gustave Eiffel, LAMA UMR8050, F-94010 Creteil, France}
\author{David Fisac}
\email{david.fisac-camara@cnrs.fr}
\address{Univ Paris Est Creteil, Univ Gustave Eiffel, CNRS, LAMA UMR8050, F-94010 Creteil, France}
\begin{document}

\begin{abstract}
We study infinite-type hyperbolic surfaces with discrete length spectrum. In this setup,
 we show that Sunada's method can only produce finite isospectral families, but that there is no bound on the cardinality of an isospectral family (for any infinite-type surface without planar ends). We then prove that, given any infinite-genus surface satisfying an additional topological assumption, any finite group can be realized as isometry group of a hyperbolic structure with discrete length spectrum.
\end{abstract}

\maketitle

\section{Introduction}

Given a hyperbolic surface, a classical invariant is its \emph{length spectrum}, that is, the multi-set of lengths of closed geodesics, counted with multiplicities. A well-known question is which aspects of the hyperbolic structure are detected by the length spectrum. Said differently, how similar are \emph{isospectral} hyperbolic surfaces (hyperbolic surfaces with the same length spectrum)?

In the closed case, it was conjectured by Gel'fand \cite{gelfand_automorphic} that the length spectrum completely determines the hyperbolic structure. The conjecture was first disproved by Vignéras \cite{vigneras_varietes}, and since then lots of work has been put into constructing isospectral families (see Gordon's survey \cite{Gordon_Survey} for more details and references). Nonetheless, the length spectrum does determine the topology, as shown by Huber \cite{huber_analytischen}. Moreover, Wolpert \cite{wolpert_length} proved that a \emph{generic} surface is determined by its spectrum. Furthermore, there is a bound on the size of an isospectral family which only depends on the topology of the surface --- see the work of Buser \cite[Chapter 13]{buser_geometry} and Parlier \cite{parlier_interrogating}. Also the \emph{simple} length spectrum (where we consider only lengths of simple closed geodesics) generically determines the hyperbolic structure (see \cite{BCK_Simple}), and analogous results have been shown for the orthospectrum and the simple orthospectrum (see \cite{MM_Systoles}, \cite{LeQuellec_Orthospectrum}).

The case of surfaces of infinite type is strikingly different. In particular, it was shown by the first author in \cite{Selfcite_Isospectral} that, for infinite-type surfaces without planar ends, there is no upper bound to the size of isospectral families. If the surface satisfies an additional topological condition (its space of ends is self-similar --- see Section \ref{sec:preliminaries}), there are isospectral families which are (uncountably) infinite.

Another aspect which clearly sets length spectra of finite- and infinite-type surfaces apart is discreteness. We say that a surface has \emph{discrete} length spectrum if, for any \(L>0\), there are finitely many closed geodesics of length at most \(L\). While the length spectrum of a finite-type hyperbolic surface is always discrete, in the infinite-type case it is very easy to construct structures with non-discrete length spectrum --- it is for instance enough to give the same length to all curves in a pants decomposition --- while it is harder to construct structures with discrete length spectrum. This is done by Basmajian and Kim in \cite{BK_Geometrically}: they show that on any infinite-type surface there are infinite-dimensional parameter spaces of hyperbolic structures with discrete (respectively, non-discrete) length spectrum.

The constructions of isospectral families in \cite{Selfcite_Isospectral} all yield hyperbolic surfaces with non-discrete length spectrum, and the goal of this paper is to understand what happens if we require the spectrum to be discrete .

Before stating our results, recall that almost all known constructions of isospectral families can be brought back to a criterion by Sunada \cite{Sun_AlmostConjugate}:

\begin{thm}[\cite{Sun_AlmostConjugate}]\label{thm:sunada}
	Let \(X\) be a complete hyperbolic surface, and \(G\) a group of isometries with finitely many fixed points. For \(H_1,H_2\) two almost conjugate subgroups of \(G\) acting without fixed points, the quotient surfaces \(X/H_1\) and \(X/H_2\) are isospectral.
\end{thm}

Note that Theorem \ref{thm:sunada} was originally stated for finite groups \(G\) and compact manifolds, but both the finiteness and compactness hypotheses can be dropped --- see \cite[Chapter 10]{buser_geometry} and the \emph{transplantation of geodesics} technique from Buser \cite{Buser_Isospectral} and Bérard \cite{berard_transplantation}.

The infinite isospectral families in \cite{Selfcite_Isospectral} are constructed using Theorem \ref{thm:sunada}. We start by asking if infinite isospectral families can be constructed in the same way in the case of surfaces with discrete length spectrum. Our first result is a negative answer to this question.

\begin{theorintro}\label{thm:finitediscretesunada}
Let \(X\) be a hyperbolic surface and \(G\leq\mathrm{Isom}^+(X)\) acting on it. Let \(\{H_i\}_{i\in I}\) be a family of almost conjugate subgroups of \(G\) without fixed points, and consider the family of hyperbolic surfaces \(\{Y_i=X/H_i\}_{i\in I}\). If the \(\Ls(Y_i)\) are discrete, then \(I\) is finite.
\end{theorintro}

On the other hand, we can show (using Sunada's criterion) that one can construct arbitrarily large isospectral families of hyperbolic surfaces with discrete length spectrum: 

\begin{theorintro}\label{thm:largefamilies}
Let \(S\) be an infinite-genus surface with no planar ends. 
Then for every positive integer \(n\), there is a hyperbolic structure \(X\) on \(S\) with discrete length spectrum and such that \(|\I(X)|\geq n\).
\end{theorintro}

Here \(\I(X)\) is the collection of all hyperbolic structures on \(S\) which are isospectral to \(X\).

One tool to prove Theorem \ref{thm:finitediscretesunada} is the fact that isometry groups of hyperbolic surfaces with discrete length spectrum are finite, as shown by Basmajian and Kim in \cite{BK_Geometrically}. Note that this is not the case without the discreteness assumption: an easy example is a translation-invariant hyperbolic structure on the Jacob's ladder surface (the surface without planar ends and with exactly two non-planar ends).

Aougab, Patel and Vlamis investigated in \cite{APV_Isometry} which groups can be realized as isometry groups of some hyperbolic structure on a given infinite-type surface, and proved the following (see Section \ref{sec:preliminaries} for the missing definitions):

\begin{thm}[\cite{APV_Isometry}]\label{thm:APV}
Let \(S\) be an infinite-genus surface without planar ends and \(G\) a group.
\begin{enumerate}
\item If the endspace of \(S\) is self-similar, there is a hyperbolic structure on \(S\) whose isometry group is \(G\) if and only if \(G\) is countable.
\item If the endspace of \(S\) is doubly pointed, then the isometry group of any hyperbolic structure on \(S\) is virtually cyclic.
\item If \(S\) contains a compact non-displaceable subsurface, there is a hyperbolic structure on \(S\) with isometry group \(G\) if and only if \(G\) is finite.
\end{enumerate}
\end{thm}

All their constructions yield surfaces with non-discrete length spectrum, so one can ask which groups can appear as isometry groups if the hyperbolic structure has discrete length spectrum. We show that, for a class of surfaces, every finite group can appear:

\begin{theorintro}\label{thm:finitegroup}
Let \(S\) be an infinite-genus surface with self-duplicating endspace and \(G\) any finite group. Then there is a hyperbolic structure \(X\) on \(S\) with discrete length spectrum and \(\Isom(X)\simeq G\).
\end{theorintro}

We say that \(S\) has \emph{self-duplicating endspace} if
\[(\Ends(S),\Ends_g(S))\simeq(\Ends(S)\times \{1,2\},\Ends_g(S)\times \{1,2\}),\]
that is, if the pair \((\Ends(S),\Ends_g(S))\) is homeomorphic to two copies of itself. An example is the \emph{blooming Cantor tree} surface, the surface without planar ends and whose endspace is a Cantor set.

If \(S\) has infinite genus, no planar ends and self-duplicating endspace, both Theorem \ref{thm:finitegroup} and Theorem \ref{thm:APV} apply. Such a surface can fall in different categories of Theorem \ref{thm:APV}: for instance, the blooming Cantor tree has self-similar endspace, so the two results show that the discreteness of the length spectrum assumption really changes the isometry groups that can arise. On the other hand, there are infinite-genus surfaces with no planar ends and self-duplicating endspace which admit a compact non-displaceable subsurface: for example, the surface without planar ends whose endspace is \(C_1\cup C_2\cup P\), where \(C_1\) and \(C_2\) are Cantor sets and \(P\) is a (countable) set of isolated ends whose closure is \(P\cup C_2\). For such surfaces, the discreteness assumption doesn't change the class of possible isometry groups. As mentioned, the constructions in \cite{APV_Isometry} all yield structures with non-discrete length spectrum, so Theorem \ref{thm:finitegroup} provides a new insight in the problem of understanding isometry groups of infinite-type surfaces also in this case.

We end this introduction with an open question. We have seen that Sunada's method doesn't allow us to construct infinite isospectral families of hyperbolic surfaces with discrete length spectrum. On the other hand, \cite{BK_Geometrically} showed that there is an abundance of surfaces with discrete length spectrum, so we ask: \textit{is there an alternative way to construct isospectral families, which yields infinite ones also in the discrete length spectrum case?}

\section*{Acknowledgements}

The first author acknowledges support of the ANR grants MAGIC (ANR-23-TERC-0007) and GALS (ANR-23-CE40-0001). The second author is funded by the ANR grant GALS (ANR-23-CE40-0001). Both authors thank Hugo Parlier for useful comments on the paper.

\section{Preliminaries}\label{sec:preliminaries}

We introduce here the basic topological, geometric and group-theoretic objects and notations needed in the rest of the paper.

The first convention is that \(\N\) represents the set of non-negative integers.

\subsection*{Topology}
By surface we mean a two-dimensional orientable connected manifold. We also assume that, if the boundary is non-empty, it is compact. As proved by Kerékjártó and Richards in \cite{Kerekjarto,Richards}, a surface \(S\) is topologically determined by its genus and number of boundary components --- which might be finite or infinite ---, and the pair of topological spaces \((\Ends(S), \Ends_g(S))\). Here, \(\Ends(S)\) is the \emph{space of ends} of the surface and \(\Ends_g(S)\) is the space of \emph{non-planar} ends (see Aramayona and Vlamis' survey \cite{AV_Survey} for the definitions).

A surface has \emph{self-similar endspace} (as first defined by Mann and Rafi in \cite{mr_large}) if for every decomposition
\[\Ends(S)=U_1\sqcup\dots\sqcup U_n\]
into pairwise disjoint clopen subsets \(U_i\), there is \(i\in\{1,\dots, n\}\) and an open subset \(A\subset U_i\) such that
\[(A,A\cap \Ends_g(S))\simeq (\Ends(S),\Ends_g(S)).\]

We say that an endspace is \emph{self-duplicating} if \((\Ends(S),\Ends_g(S))\) is homeomorphic to two copies of itself, i.e.\
\[(\Ends(S),\Ends_g(S))\simeq (\Ends(S)\times \{1,2\},\Ends_g(S)\times \{1,2\}).\]
Note that this can be easily seen to be equivalent to \((\Ends(S),\Ends_g(S))\) being homeomorphic to \(n\) copies of itself, for every \(n\geq 1\).

An endspace is \emph{doubly pointed} (see \cite{APV_Isometry}) if there are exactly two ends with finite homeomorphism group orbit.

Note that:
\begin{itemize}
\item if an endspace is self-similar or self-duplicating, it is not doubly pointed;
\item endspaces can be both self-similar and self-duplicating (e.g.\ in the case of a blooming Cantor tree), self-duplicating but not self-similar (e.g.\ if \(\Ends(S)=\Ends_g(S)=C_1\cup C_2\cup P\), where \(C_1\), \(C_2\) and \(P\) are pairwise disjoint, \(C_1\) and \(C_2\) are Cantor sets and \(P\) is a set of isolated points whose closure is \(C_2\cup P\)), or self-similar but not self-duplicating (e.g.\ if \(|\Ends(S)|=|\Ends_g(S)|=1\)).
\end{itemize}

\begin{figure}[h!]
	\label{fig:BloomCant}
	\centering
	\includegraphics[width=300pt]{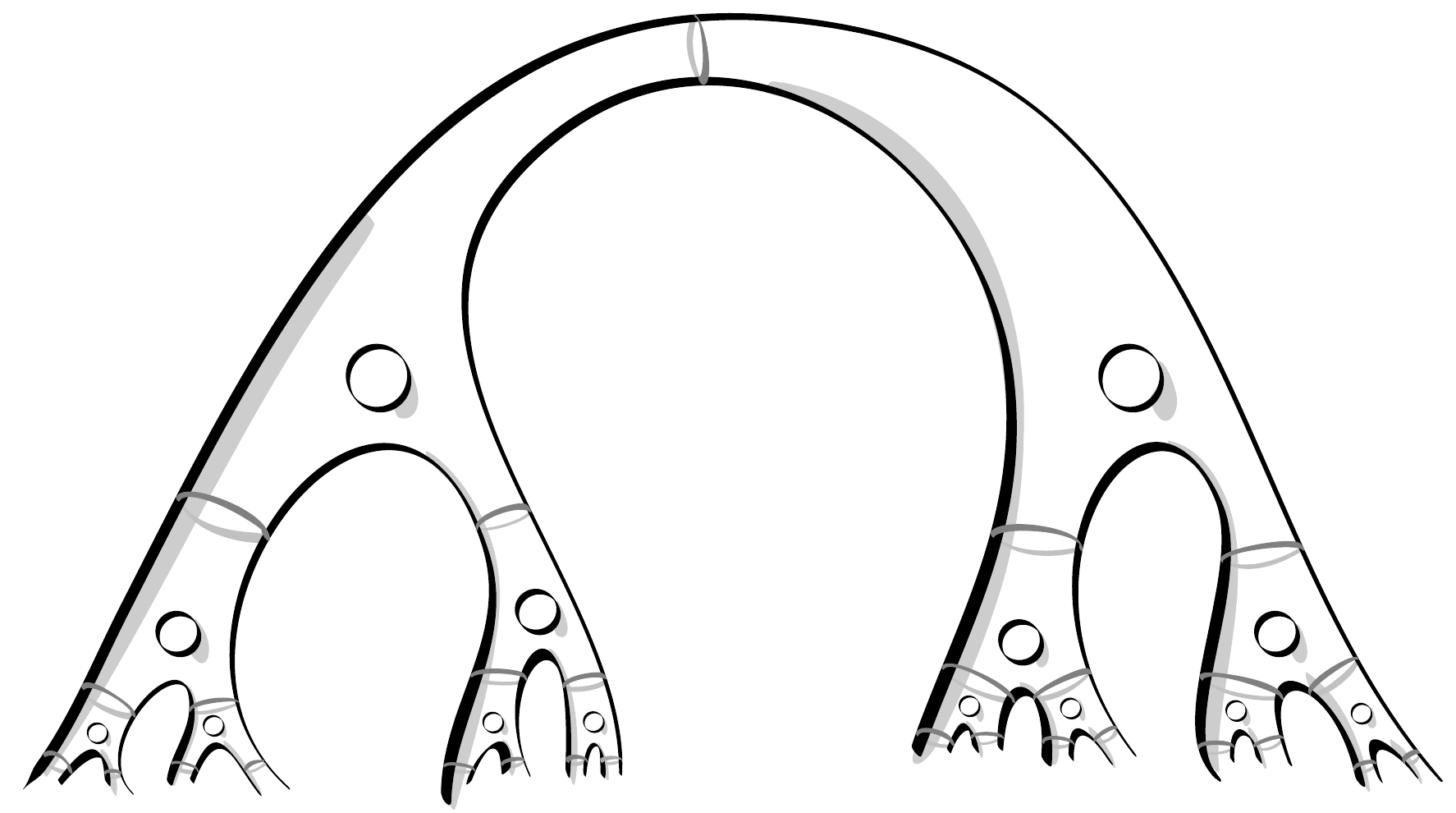}
	\caption{Blooming Cantor tree surface.}
\end{figure}

A compact subsurface \(K\) of a surface \(S\) is \emph{non-displaceable} if for every \(f\in\mathrm{Homeo}(S)\), \(f(K)\cap K\neq \emptyset\).

\subsection*{Geometry}

A \emph{hyperbolic structure} \(X\) on a surface \(S\) is a Riemannian metric of constant curvature \(-1\). We will assume such a metric to be complete and of the first kind, if \(\partial S=\emptyset\), otherwise we ask that each boundary component is totally geodesic and that the metric on the double is complete and of the first kind. We will also call \(X\) a \emph{hyperbolic surface}. Define the \emph{Teichmüller space} of a surface \(S\), denoted \(\Teich(S)\), to be the space of \emph{marked} hyperbolic structures on \(S\), i.e.\ the space of hyperbolic structures up to isotopy. The \emph{moduli space} of \(S\), denoted by \(\ms(S)\), is the space of hyperbolic structures on \(S\) up to isometry. Its universal cover is Teichm\"uller space.

Given a hyperbolic surface \(X\), define its \emph{length spectrum} \(\Ls(X)\) to be the multi-set of lengths of closed geodesics on the surface with multiplicity. We call two hyperbolic surfaces \emph{isospectral} if their length spectra coincide. Given a hyperbolic surface \(X\) with underlying topological surface \(S\), we denote by \(\I(X)\) the set
\[\I(X):=\{Y\in\ms(S)\st \Ls(Y)=\Ls(X)\}\]
and by \(\tilde{\I}(X)\) the lift of \(\I(X)\) to Teichmüller space.

In our constructions we will make repeated use of the so-called Collar Lemma (see e.g.\ \cite[Chapter 4]{buser_geometry}):

\begin{thm}[Collar Lemma]
Let \(X\) be a hyperbolic surface and \(\alpha\) a simple closed geodesic on it. Then \(X\) contains an embedded collar around \(alpha\) of width \(\arcsinh\left(\frac{1}{\sinh(\ell_\alpha(X)/2)}\right)\).
\end{thm}

A consequence is that any two simple closed geodesics of length at most \(\arcsinh(1)\) are disjoint.

We also remark the following:

\begin{lemma}\label{lem:quotient-discrete}
Let \(X\) be a hyperbolic surface and \(G\) a finite group of isometries without fixed points. If the length spectrum of \(X/G\) is discrete, the same holds for the length spectrum of \(X\).
\end{lemma}

\begin{proof}
If the spectrum of \(X\) were not discrete, we could find a sequence of distinct closed geodesics \(\gamma_n\) of \(X\) with length uniformly bounded from above. Their projections in \(X/G\) will also have a uniform length upper bound, and for every \(n\), the projection of \(\gamma_n\) can coincide with the projection of at most \(|G|\) other \(\gamma_m\). So we also get infinitely many distinct closed geodesic on \(X/G\) with length bounded uniformly from above, i.e.\ \(\Ls(X/G)\) is not discrete either.
\end{proof}

\subsection{Group theory}
The main group-theoretic notion we will use is almost-conjugacy. Given a group \(G\) and two subgroups \(H_1,H_2\leq G\), we say that \(H_1\) and \(H_2\) are \emph{almost conjugate} if, for every \(g\in G\), 
\[
|C_g\cap H_1|=|C_g\cap H_2|,
\]  
where \(C_g\) denotes the conjugacy class of \(g\) in \(G\).

It is not hard to show that if \(H_1\) and \(H_2\) are almost conjugate subgroups of a group \(H\), then for any \(n\geq 1\) and any \(\sigma_1,\sigma_2:\{1,\dots, n\}\to\{1,2\}\), \(\prod_{i=1}^n H_{\sigma_1(i)}\) and \(\prod_{i=1}^n H_{\sigma_2(i)}\) are almost conjugate subgroups of \(H^n\).

We will use a specific triple \(H, H_1, H_2\), described in \cite[Example 11.2.2]{buser_geometry}: \(H=(\Z/8\Z)^*\ltimes \Z/8\Z\), where \((a,b)\cdot(a',b')=(aa',ab'+b)\), and the subgroups are
\[H_1=\{e=(1,0),h_1=(3,0),h_2=(5,0),(7,0)\}\]
and
\[H_2=\{(1,0),h_3=(3,4),h_4=(5,4),(7,0)\}.\]
Then \(H_1\) and \(H_2\) are almost conjugate, but not conjugate.

\begin{rmk}\label{rmk:groupstuff}Since $h_1$ and $h_2$ generate $H_1$, for every $h\in H$, $\{hh_1,hh_2\}\not\subseteq H_2h$. Similarly, for every $h\in H$, $\{hh_3,hh_4\}\not\subseteq H_1h$
\end{rmk}

\section{Finiteness of isospectral families given by Sunada's construction}

In this section we will prove Theorem \ref{thm:finitediscretesunada}, which we recall here:

\begin{thm:finitediscretesunada}
	Let \(X\) be a hyperbolic surface and \(G\leq\Isom^+(X)\) acting on it. Let \(\{H_i\}_{i\in I}\) be a family of almost conjugate subgroups of \(G\) without fixed points. Define a family of hyperbolic surfaces \(\{Y_i=X/H_i\}_{i\in I}\). If \(\Ls(Y_i)\) are discrete, then \(I\) is finite.
\end{thm:finitediscretesunada}

We will use two lemmas in the proof, the first one being purely group-theoretic. 

\begin{lemma}\label{lem:finacsubgroups}
	Let \(G\) be a group and \(\{H_i\}_{i\in I}\) be a family of almost conjugate subgroups of finite index. Then \(I\) is finite.
\end{lemma} 
\begin{proof}
	By definition, the subgroups are almost conjugate if, for every conjugacy class \(C\) of \(G\), then 
	\[
		|H_i\cap C|=|H_j\cap C|
	\]
	for all \(i,j\in I\). It is a standard fact (see e.g.\ \cite[Proposition 4.1C]{Gordon_Survey}) that their permutation representations 
	\begin{align*}
	\rho_{H_i}:\, &G\to \mathrm{Sym}(G/H_i)\\
	&g\mapsto (h\mapsto g\cdot h)	
	\end{align*}
	are pairwise conjugate for all \(i\in I\). Therefore, by elementary computations, their kernels coincide. Denote then \[K=\ker(\rho_{H_i})=\bigcap_{g\in G}gH_ig^{-1}\] for all $i\in I$. Note also that \(K\) is a normal subgroup of all \(H_i\)'s. By considering \(H_i/K\leq G/K\), it is straightforward that \([G:K]\leq [G:H_i]<\infty\). Since the subgroups of \(G\) containing \(K\) are in bijection with the subgroups of \(G/K\), which is finite, we conclude.
\end{proof}
\begin{lemma}\label{lem:finstab}
	Let \(X\) be a hyperbolic surface and \(G<\Isom^+(X)\) acting on it. Then,
	\[
		|\mathrm{Stab}_G(\gamma)|<\infty
	\]
	for all closed geodesics \(\gamma\).
\end{lemma}
\begin{proof}
	 Note that \(\mathrm{Stab}_G(\gamma)\subseteq \{f\in\Isom(X)\mid f(\gamma)\cap \gamma\neq\emptyset\}\). Then the result follows from proper discontinuity of the action of \(\Isom(X)\) (see e.g.\ \cite[Lemma~2.6]{APV_Isometry}).
\end{proof}

\begin{proof}[Proof of Theorem \ref{thm:finitediscretesunada}]
	Assume first that \(\Ls(X)\) is discrete. By \cite[Prop~3.4]{BK_Geometrically}, then \(\Isom(X)\) is finite, and hence \([G:H_i]\) are also all finite. 
	In this case the result follows from Lemma~\ref{lem:finacsubgroups}.
	
	Assume now that \(\Ls(X)\) is not discrete. We will now prove that in this case \([G:H_i]\) are still all finite. It is enough to prove it for one of the \(H_i\)'s, since almost conjugate subgroups have the same index. Take \(H\) to be one of them for simplicity.
	
	Fix a closed geodesic \(\gamma\) on \(X\). Note that there is a natural correspondence between the orbit set of \(\gamma\), \(G\cdot \gamma\), and the quotient group \(G/\mathrm{Stab}_G(\gamma)\). We will define two quotients on the quotient set. First, 
	\[
		A_H(\gamma)=G\cdot \gamma /\sim\,, \text{ where }\alpha\sim \beta \text{ if }h(\alpha)=\beta\text{ for some }h\in H.
	\] 
	Note that \(A_H(\gamma)\) is finite for all \(\gamma\), because of \(\Ls(X)\) being non-discrete and \(\Ls(X/H)\) being discrete. Note now that \(A_H(\gamma)\) is also in bijection with the group \(G/\approx\), where \(g_1\approx g_2\) if \(h(g_1(\gamma))=g_2(\gamma)\) for some \(h\in H\). That is equivalent to \(g_1\in \mathrm{Stab}_G(\gamma)g_2 H\). There is a natural bijection between \(A_H(\gamma)\) and \(\mathrm{Stab}_G(\gamma)\backslash G/H\). Furthermore, by definition, 
	\[
		[G:H]= \sum_{\mathrm{Stab}_G(\gamma)g H\in\mathrm{Stab}_G(\gamma)\backslash G/H }[\mathrm{Stab}_G(\gamma):\mathrm{Stab}_G(\gamma)\cap gHg^{-1}].
	\]
	Note that, since \(A_H(\gamma)\) is finite, so is the index set of the sum, and that every term of the sum is finite because of Lemma~\ref{lem:finstab}. Hence, \([G:H]\) is finite, and therefore we can conclude again by Lemma~\ref{lem:finacsubgroups}.
\end{proof}

\section{Arbitrarily large isospectral families}
In this section we will prove Theorem \ref{thm:largefamilies}.

\begin{thm}\label{thm:largefamilies-general}
Let \(S\) be an infinite-genus surface without planar ends. For every \(n\in \N\) there is a hyperbolic structure \(X\) on \(S\) with discrete length spectrum and such that \(|\I(X)|\geq n\).
\end{thm}

\begin{proof}
Consider \(H=(\mathbb{Z}/8\mathbb{Z})^*\ltimes\mathbb{Z}/8\mathbb{Z} \), \(H_1=\{(1,0),h_1=(3,0),h_2=(5,0),(7,0)\}\) and \(H_2=\{(1,0),(3,4),(5,4),(7,0)\}\) as in Section \ref{sec:preliminaries}.

Given a surface \(S\) and a positive integer \(n\), let \(G\) be the product of \(n\) copies of \(H\) and denote by \(K_i\), for \(i=1,\dots, n\), the subgroup \(H_1\times\dots H_1\times H_2\times H_1 \times \dots \times H_1\), where \(H_2\) is at the \(i\)-th position. Fix a countable dense subset \(F\) of \(\Ends(S)\), and pick a collection of pairwise disjoint nonseparating curves \(\{\alpha^e_n\st n\in \N, e\in F\}\) such that for every \(e\in F\), \(\alpha^e_n\) converges to \(e\) as \(n\to\infty\). 

Up to possibly changing the collection of curves, we can assume that there are \(e_0\in F\) and simple closed curves \(\mu_0,\dots,\mu_N\), pairwise disjoint and disjoint from all the \(\alpha^e_n\), such that:
\begin{itemize}
\item for every \(i\leq N-1\), \(\mu_i\) cuts of a one-holed torus containing \(\alpha^{e_0}_i\);
\item \(\mu_N\) is disjoint from all the \(\mu_i\), \(i\leq N-1\), and \(\mu_0,\dots,\mu_N\) cobound an \(N\)-holed sphere.
\end{itemize}

\hfill 

\begin{figure}[h!]
	\label{fig:foot}
	\begin{overpic}[width=280pt]{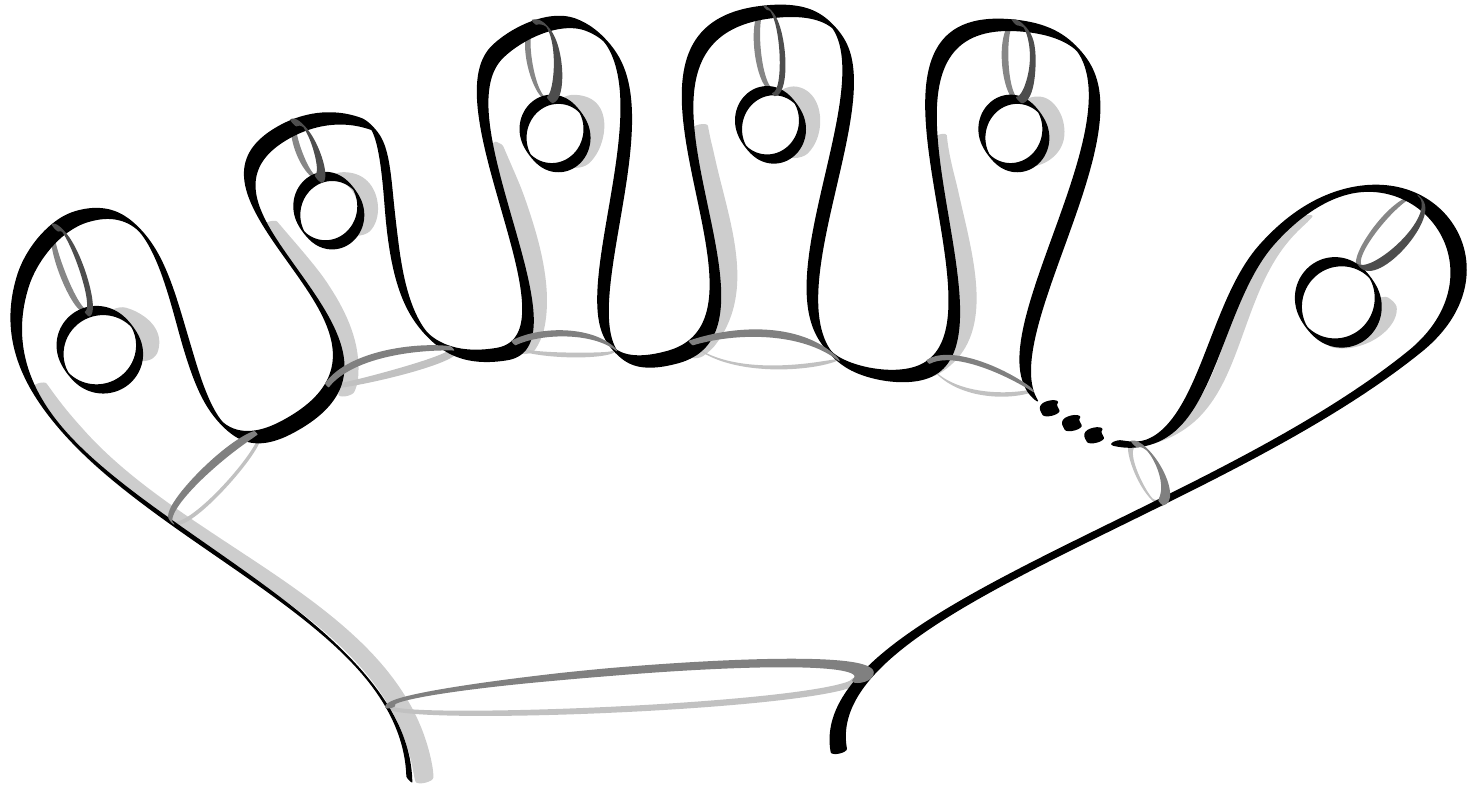}
		\put(2,40.5){$\alpha^{e_0}_0$}
		\put(18,47.5){$\alpha^{e_0}_1$}
		\put(35,54){$\alpha^{e_0}_2$}
		\put(51,54.5){$\alpha^{e_0}_3$}
		\put(68,53.5){$\alpha^{e_0}_4$}
		\put(95,41){$\alpha^{e_0}_{N-1}$}
		\put(15,19){$\mu_0$}
		\put(26,26){$\mu_1$}
		\put(37,27){$\mu_2$}
		\put(50,26.5){$\mu_3$}
		\put(64.5,25){$\mu_4$}
		\put(69,19.5){$\mu_{N-1}$}
		\put(41,10){$\mu_N$}
	\end{overpic}
	\centering
	\caption{Construction of the curves \(\alpha^{e_0}_i\) and \(\mu_i\).}
\end{figure}

Let \(V\) be the surface obtained by cutting \(S\) along all the \(\alpha^e_n\), and for every \(n,e\) we denote by \(\alpha^e_{n,\pm}\) the two boundary components of \(V\) corresponding to \(\alpha^e_{n}\).

Let \(G=\{g_0,\dots, g_{N-1}\}\), where \(N=32^n\) is the cardinality of \(G\). We rename the boundary components of \(V\) as \(\{\partial(e,r,q,\pm)\st e\in F,r\in\{0,\dots, N-1\}, q\in \N\}\), by setting
\[\partial(e,r,q,\pm):=\alpha^e_{qN+r,\pm}.\]

We build a surface \(X\) as follows: for every \(h\in G\), let \(V_h\) be a copy of \(V\). For every \(h\in G\), \(r\in \{0,\dots, N-1\}\), \(q\in \N\) and \(e\in F\), glue the boundary component \(\partial(e,r,q,+)\) of \(V_h\) to the boundary component \(\partial(e,r,q,-)\) of \(V_{hg_r}\). We consider the action of \(G\) on \(X\) obtained by rigidly permuting the subsurfaces \(V_h\). Our first claim is that \(X\) and \(X/K_i\), for any \(i\in\{1,\dots, n\}\), are homeomorphic to \(S\).

To prove the claim, note first that \(X\) and \(X/K_i\) are infinite-genus surfaces without boundary and with no planar ends. So to show that they are homeomorphic to \(S\) it is enough to prove that they all have homeomorphic spaces of ends.

Note that \(X/G\) is homeomorphic to \(S\). Moreover, for every \(i\in \{1,\dots n\}\), we have natural quotient maps
\[X\twoheadrightarrow X/K_i\twoheadrightarrow X/G\simeq S,\]
which induce surjective continuous maps
\[\Ends(X)\twoheadrightarrow \Ends(X/K_i)\twoheadrightarrow \Ends(S).\]
As all spaces of ends are compact and Hausdorff, to prove that the maps are homeomorphisms it is enough to show that they are injective. In particular, it is enough to prove that the map \(\Ends(X)\to\Ends(S)\) is injective. This holds because \(G\) acts trivially on the endspace of \(X\), which can be shown as in \cite[Lemma 3.5]{APV_Isometry}.

We endow \(S\) with a hyperbolic structure with the following properties:
\begin{enumerate}
\item the length spectrum is discrete;
\item there is \(\varepsilon<\arcsinh(1)\) such that all closed geodesics of \(V\) have length bigger than \(\varepsilon\), except for \(\{\mu_0,\dots,\mu_N\}\) and \(\{\alpha^{e_0}_0,\dots, \alpha^{e_0}_{N-1}\}\), whose lengths are all distinct and in \((0,\varepsilon)\).
\end{enumerate}
To construct such a structure, it is enough to start from a hyperbolic structure with discrete length spectrum (provided for instance by \cite{BK_Geometrically}), pick \(\varepsilon<\arcsinh(1)\) and smaller than the systole (which is possible by discreteness of the length spectrum), and then pinch \(\{\mu_0,\dots,\mu_N\}\) and \(\{\alpha^{e_0}_0,\dots, \alpha^{e_0}_{N-1}\}\) until they reach the correct length.

The metric on \(S\) induces a hyperbolic structure on \(V\), and on \(X\) by gluing copies of \(\alpha^e_{n,\pm}\) using the twist parameters of \(\alpha^e_n\) on \(V\). So \(G\) is realized as an isometry group of \(X\), without fixed points, and thus \(X\) and \(X/K_i\) all have discrete length spectrum by Lemma \ref{lem:quotient-discrete}. By Theorem \ref{thm:sunada}, the \(X/K_i\) are pairwise isospectral. We only need to show that they are not isometric.

For ease of notation, let us show that \(X/K_1\) and \(X/K_2\) are not isometric; the argument is analogous for any two quotients. Suppose for simplicity that \(g_0=(h_3,(1,0),\dots, (1,0))\) and \(g_1=(h_4,(1,0),\dots,(1,0))\). In \(X/K_1\), in the copy \(V_{K_1}\) of \(V\) corresponding to the coset \(K_1\), \(\partial(e_0,0,0,+)\) is glued to \(\partial(e_0,0,0,-)\) and \(\partial(e_0,0,1,+)\) is glued to \(\partial(e_0,0,1,-)\). So we can find a simple closed geodesic \(\gamma\subset V_{K_1}\) such that:
\begin{itemize}
\item it intersects the curve corresponding to \(\partial(e_0,0,0,+)\) once,
\item it intersects the curve corresponding to \(\partial(e_0,0,1,+)\) once,
\item it intersects \(\mu_0\) twice,
\item it intersects \(\mu_1\) twice,
\item it is disjoint from \(\mu_N\).
\end{itemize}
\begin{figure}[h!]
	\label{fig:isometric}
	\begin{overpic}[width=130pt]{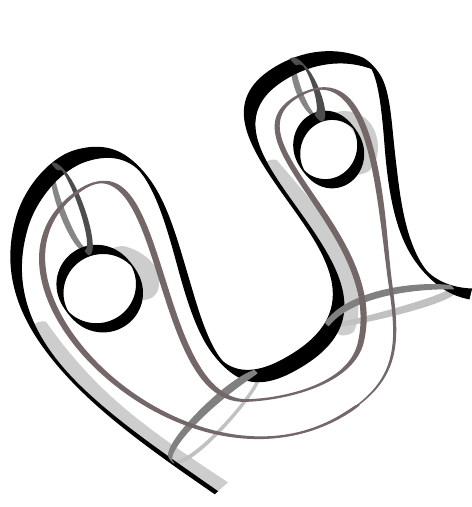}
		\put(-7,73){\(\partial(e_0,0,0,\pm)\)}
		\put(39,92){\(\partial(e_0,0,1,\pm)\)}
		\put(22,7){\(\mu_0\)}
		\put(56,38){\(\mu_1\)}
		\put(70,14){\(\gamma\)}
	\end{overpic}
	\centering
	\caption{Construction of the curve \(\gamma\).}
\end{figure}
If \(\varphi:X/K_1\to X/K_2\) is an isometry, the curve \(\varphi(\gamma)\) must intersect a copy of \(\partial(e_0,0,0,+)\) in some copy of \(V\) once, a copy of \(\partial(e_0,0,1,+)\) in some copy of \(V\) once, a copy of \(\mu_0\) in some copy of \(V\) twice, a copy of \(\mu_1\) in some copy of \(V\) twice and be disjoint from all copies of \(\mu_N\). One can check that this implies that there is some copy of \(V\) where  \(\partial(e_0,0,0,+)\) is glued to \(\partial(e_0,0,0,-)\) and \(\partial(e_0,0,1,+)\) is glued to \(\partial(e_0,0,1,-)\). In algebraic terms, this means that there is some \(g\in G\) such that
\[K_2gg_0=K_2g\]
and
\[K_2gg_1=K_2g.\]
Looking at the first component, this means that there is \(h\in H\) so that \(hh_3, hh_4 \in H_1h\), which is impossible (see Remark \ref{rmk:groupstuff}).
\end{proof}

\section{Realizing finite groups as isometry groups}

The goal of this section is to prove that all finite groups can be realized as isometry groups of hyperbolic structures with discrete length spectrum on a large class of surfaces:

\begin{thm:finitegroup}
Let \(S\) be an infinite-genus surface with self-duplicating endspace and \(G\) any finite group. Then there is a hyperbolic structure \(X\) on \(S\) with discrete length spectrum and \(\Isom(X)\simeq G\).
\end{thm:finitegroup}

\begin{rmk}
Note that every hyperbolic structure with discrete length spectrum has finite isometry group (see \cite[Proposition 3.4]{BK_Geometrically}), so our result completely classifies the possible isometry groups of hyperbolic structures with discrete length spectrum on an infinite-genus surface with self-duplicating endspace.
\end{rmk}

We will use the following result, which follows from the work of Basmajian and Kim \cite{BK_Geometrically}:

\begin{lemma}\label{lem:discrete-with-boundary}
Let \(S\) be a surface with a single boundary component and let \(\varepsilon>0\). Then there is a hyperbolic structure \(X\) on \(S\) with discrete length spectrum and such that the boundary component has length \(\varepsilon\).
\end{lemma}

\begin{proof}
It is enough to observe that the construction by Basmajian and Kim \cite{BK_Geometrically} can be easily adapted to include one geodesic boundary component, and that one can freely choose its length \(\varepsilon\). 
\end{proof}

\begin{rmk}\label{rmk:discreteorthospectrum}
The orthospectrum of any (finite- or infinite-type) hyperbolic surface with compact boundary is discrete. One way of seeing this is to double the surface along the boundary and use the fact that the multiset of length of geodesics intersecting a given compact set (in this case, what used to be the boundary) is discrete. 
\end{rmk}

\begin{proof}[Proof of Theorem \ref{thm:finitegroup}]
Fix \(S\) as in the statement, a finite group \(G\) and \(\varepsilon\in(0,\arcsinh(1))\). Let \(Z\) be a hyperbolic surface with a single boundary component of length \(\varepsilon\), infinite genus, endspace \((\Ends(Z),\Ends_g(Z))\simeq (\Ends(S),\Ends_g(S))\) and discrete length spectrum (which exists by Lemma \ref{lem:discrete-with-boundary}). By discreteness of \(\Ls(Z)\) and by Remark \ref{rmk:discreteorthospectrum}, there is \(\delta>0\) such that every closed geodesic and every orthogeodesic in \(Z\) has length bigger than \(\delta\). We will choose \(\delta<\varepsilon\).

Let \(V\) be a \((2|G|+1)\)-holed sphere, with boundary components denoted \(\{\gamma_g, \bar{\gamma}_g, \gamma\st g\in G\}\). Fix a hyperbolic structure on \(V\) by choosing a pants decomposition \(\mathcal{P}\), assigning length via an injective function \(\lambda:\mathcal{P}\cup \{\gamma_g, \bar{\gamma}_g\st g\in G\}\to (0,\delta)\), setting the length of \(\gamma\) to be \(\varepsilon\) and choosing all twist parameters to be zero.

We then construct a compact hyperbolic surface \(F\) modelled on the Cayley graph of \(G\) with \(G\) as generating set, using \(V\) as vertex surface, as in \cite{Allcock_Hyperbolic}. Precisely, for every \(g\) in \(G\), let \(V_g\) be a copy of \(V\). For any \(g,h\in G\), glue the \(\gamma_h\) boundary component of \(V_g\) to the \(\bar{\gamma}_h\) boundary component of \(V_{gh}\), with zero twist. As in \cite{Allcock_Hyperbolic}, the resulting surface has \(\Isom(F)=G\), where \(G\) acts by rigidly permuting the \(V_g\)'s. Denote by \(\mathcal{Q}\) the collection of curves in \(F\) obtained as union of all the curves in \(\mathcal{P}\) for every vertex surface \(V_g\). Note that these are exactly the simple closed geodesics in \(F\) of length at most \(\delta\) (by the Collar Lemma).

By construction, \(F\) has \(|G|\) boundary components, all of length \(\varepsilon\) (one per surface \(V_g\)). For every \(g\in G\), glue a copy \(Z_g\) of \(Z\) to the \(\gamma\) boundary component of \(V_g\),  with zero twist, to get a surface \(X\). 
\begin{figure}[h!]
	\label{fig:example}
	\begin{overpic}[width=280pt]{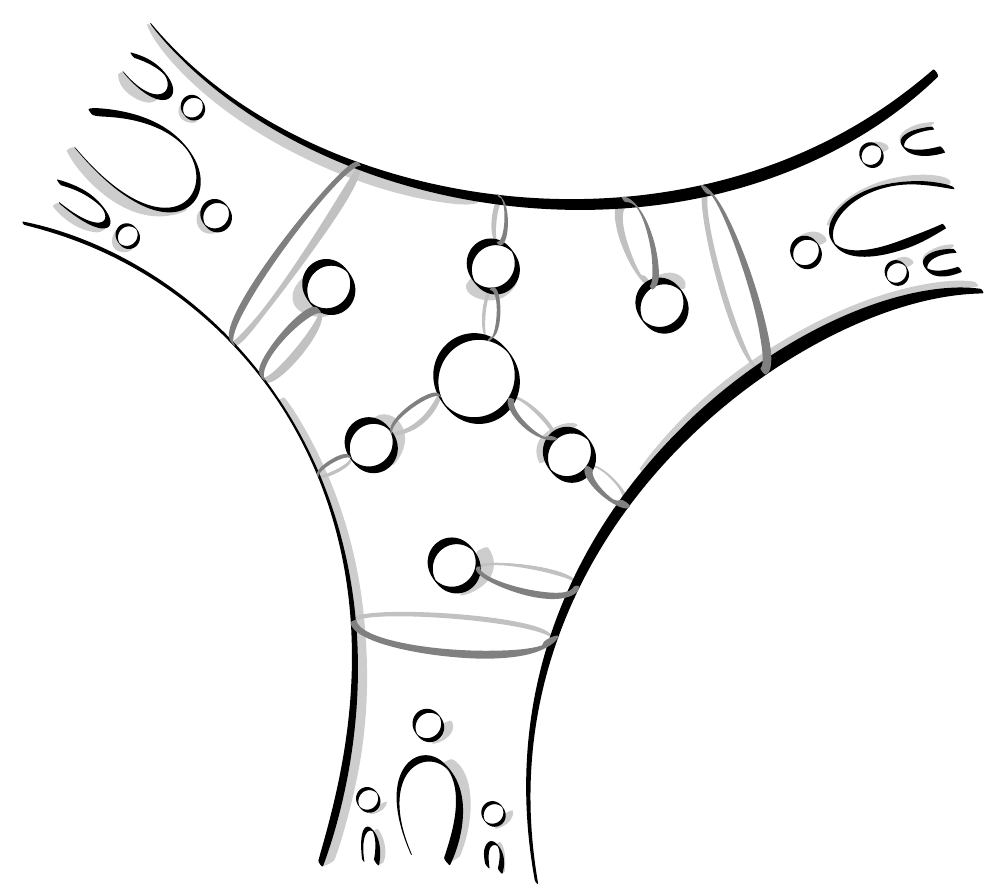}
		\put(37.5,55.5){$V_0$}
		\put(45,40){$V_1$}
		\put(57,54){$V_2$}
		\put(25.5,71){$Z_0$}
		\put(38,20){$Z_1$}
		\put(78,59){$Z_2$}
		\put(24.5,63){\footnotesize$\gamma$}
		\put(25,56){\footnotesize$\gamma_0$}
		\put(28,53){\footnotesize$\bar{\gamma}_0$}
		\put(39.5,50){\footnotesize$\bar\gamma_1$}
		\put(41.5,46.5){\footnotesize$\gamma_1$}
		\put(31,45){\footnotesize$\bar\gamma_2$}
		\put(33,41){\footnotesize$\gamma_2$}
		\put(46,58){\footnotesize$\gamma_1$}
		\put(51,57.5){\footnotesize$\bar{\gamma}_1$}
		\put(51.5,66.5){\footnotesize$\bar{\gamma}_2$}
		\put(46,67){\footnotesize$\gamma_2$}
	\end{overpic}
	\centering
	\caption{Construction of \(X\) for \(G=\mathbb{Z}/3\mathbb{Z}\) and \(S\) the blooming Cantor tree surface.}
\end{figure}

The action of \(G\) extends to an isometric action on \(X\) (by rigidly permuting the \(Z_g\)), so \(G<\Isom(X)\).

Note that the curves in \(\mathcal{Q}\) are the only simple closed gedesics in \(X\) of length at most \(\delta\): indeed, consider a simple closed geodesic \(\alpha\). If it is contained in \(F\), we have already shown that \(\alpha\in\mathcal{Q}\) if and only if its length is less than \(\delta\). If it is not contained in \(F\), there is a \(Z_g\) which either contains \(\alpha\) or such that \(\alpha\cap Z_g\) contains an arc from the boundary to the boundary. Since all closed geodesics and orthogeodesics of \(Z\) have length more than \(\delta\), \(\alpha\) must have length more than \(\delta\).

As a consequence, any isometry \(\varphi\) of \(X\) needs to send \(\mathcal{Q}\) to itself. If \(\varphi(F)\) were different from \(F\), there would be a boundary component of \(F\) sent to a curve intersecting the boundary of \(F\), which is impossible by the Collar Lemma. So \(\varphi\) restricts to an isometry of \(F\), and thus agrees on \(F\) with some \(g\in G\). As two isometries agreeing on an open set are the same map, \(\varphi=g\), and hence \(\Isom(X)=G\).

By construction \(X\) has infinite genus and its endspace is \(\sqcup_{g\in G}(\Ends(Z_g),\Ends_g(Z_g))\). As \(S\) has self-duplicating endspace, this is homeomorphic to  \((\Ends(S),\Ends_g(S))\). So \(X\) is homeomorphic to \(S\).

Finally, we want to show that \(\Ls(X)\) is discrete. By contradiction, suppose that \(\{\alpha_n\st n\in \N\}\) is a collection of closed geodesics of uniformly bounded length. By discreteness of \(\Ls(Z)\) and of \(\Ls(F)\), they cannot be all contained in any \(Z_g\) or in \(F\). If there is a \(Z_g\) and a subsequence \(n_j\) such that for every \(j\), \(\alpha_{n_j}\cap Z_g\) contains an arc \(a_j\) from the boundary to the boundary of \(Z_g\) and the \(a_j\) are pairwise not homotopic, by discreteness of the orthospectrum of \(Z\) the lengths of the \(a_j\) go to infinity, a contradiction. So the only possibility is that all the \(\alpha_n\) intersect the \(F\) and all the \(Z_g\) in the same arcs (up to homotopy relative to the boundary of \(F\) or \(Z_g\)). But then the \(\alpha_n\) differ by Dehn twists about the boundary components of \(F\), with larger and larger powers (in absolute value). In particular their lengths cannot be uniformly bounded, a contradiction.
\end{proof}

\begin{rmk}
Note that the construction in the theorem could be used to produce arbitrarily large isospectral families for surfaces with self-duplicating space of ends, by applying Sunada's criterion with the same groups used in the proof of Theorem \ref{thm:largefamilies}.
\end{rmk}

\section{Limits of isospectral families}

One of the reasons why discreteness of the length spectrum is an interesting assumption is that it is one of the crucial ingredient in Wolpert's proof of finiteness of isospectral families in the closed case \cite{wolpert_length}. The sketch of the proof is the following: given a sequence of isospectral surfaces in a moduli space, note that it sits in a compact set --- by Mumford's criterion \cite{mumford_remark}, since all surfaces in the sequence have the same systole. So we can extract a converging subsequence, and lift it to a converging subsequence in Teichm\"uller space. By discreteness of the length spectrum, and using the fact that the lengths of finitely many curves determine a point in Teichmüller space (see e.g.\ the \(9g-9\) theorem \cite[Theorem 10.7]{fm_primer}), Wolpert then shows that the sequence is eventually constant. In particular, any limit of isospectral surfaces is isospectral to all surfaces in the sequence.

As discussed, in the infinite-type setting infinite isospectral families exist, and indeed the main tools used by Wolpert (Mumford's criterion, the \(9g-9\) theorem and discreteness of the length spectrum) are missing in this setup. In attendance of possible future constructions of infinite families of discrete isospectral structures, if we do assume discreteness, we can try and see which partial results can be recovered. We end this article by showing that discreteness allows us to show that the limit of isospectral surfaces also has the same spectrum.

\begin{prop}\label{prop:limit-isospectral}
Let \(X\) be a hyperbolic surface with discrete length spectrum. Suppose \(Y_k\in\I(X)\) is a converging sequence, with limit \(Y\). Then \(Y\in\I(X)\).
\end{prop}

To prove this, we will use the following lemma:

\begin{lemma}
Suppose \(\Ls(X)\) is discrete and \(Y_k\in\tilde{\I}(X)\) is a converging sequence, with limit \(Y\). Let \(L>0\). Then there is \(k_0\geq 1\) so that for every curve \(\gamma\) with \(\ell_Y(\gamma)\leq L\) and for every \(k\geq k_0\), \(\ell_{Y_k}(\gamma)=\ell_Y(\gamma)\).
\end{lemma}

\begin{proof} Fix \(L\); by discreteness, there is \(\varepsilon>0\), which we assume smaller than \(L\), such that every interval \(I\subset [0,2L]\) of length at most \(\varepsilon\) contains at most one value of \(\Ls(X)\).

Since \(Y_k\) converges, there is \(k_0\) so that for every \(k\geq k_0\), \(\ell_{Y_k}(\gamma)\in (\ell-\varepsilon/2,\ell+\varepsilon/2)\cap [0,2L]\), where \(\ell=\ell_Y(Y_k)\). But then by the previous observation \(\ell_{Y_k}(\gamma)\) is the unique element of \(\Ls(X)\) in the given interval, and thus the sequence is constant for \(k\geq k_0\). Since \(Y_k\to Y\), the constant can only be \(\ell_Y(\gamma)\). 
\end{proof}

\begin{proof}[Proof of Proposition \ref{prop:limit-isospectral}]
Lift the sequence \(Y_k\) and the limiting point \(Y\) to a sequence in \(\tilde{\I}(X)\) and limiting point in Teichmüller space. Abusing notation, we call the lifted surfaces \(Y_k\) and \(Y\).

Let \(\ell\in\Ls(Y)\). Then there is \(\gamma\) with \(\ell_Y(\gamma)=\ell\), and by the proposition \(\ell_{Y_k}(\gamma)\) is eventually constant and equal to \(\ell\). Thus \(\ell\in\Ls(Y_k)=\Ls(X)\) for every sufficiently large \(k\). In particular, the set of lengths (forgetting multiplicities) of closed geodesics in \(Y\) is contained in that of \(X\). To show isospectrality it is then enough to show that each \(\ell\in\Ls(X)\) appears in \(\Ls(Y)\) with the correct multiplicity.

Consider \(\ell\in \Ls(X)\), and suppose its multiplicity is \(m>0\). For every \(k\), there are exactly \(m\) curves \(\gamma_1^k,\dots,\gamma_m^k\) such that
\[\ell_{Y_k}(\gamma_i^k)=\ell\]
for every \(i\). Note that, for every sufficiently large \(k\), \(\ell_Y(\gamma_i^k)\leq 2\ell=:L\). Let \(k_0\) be given by the lemma; then for every \(k\geq 0\)
\[\ell_Y(\gamma_i^{k_0})=\ell_{Y_{k_0}}(\gamma_i^{k_0})=\ell_{Y_k}(\gamma_i^{k_0})=\ell\]
for every \(i=1,\dots, m\). So \(\ell\in \Ls(Y)\) with multiplicity at least \(m\), and moreover (up to renumbering)
\[\gamma_i^k=\gamma_i^{k_0}\]
for every \(k\geq k_0\). So \(\ell\in\Ls(Y)\) and has multiplicity at least \(m\).

Now suppose that \(\gamma\) is a curve with \(\ell_Y(\gamma)=\ell\). By the lemma, \(\ell_{Y_k}(\gamma)\) is eventually constant and equal to \(\ell_Y(\gamma)=\ell\). So \(\gamma\) is one of the \(\{\gamma_i^{k_0}\st i=1,\dots, m\}\). This implies that \(\ell\) has multiplicity \(m\) in \(\Ls(Y)\).
\end{proof}

\bibliographystyle{alpha}
\bibliography{references}

\end{document}